\documentclass{amsart}

\usepackage{amsmath,amsthm,amssymb}
\usepackage{graphicx}
\usepackage{tikz-cd}
\usepackage{paralist}
\usepackage{environ}
\usepackage{xcolor}
\usepackage{hyperref}
\usepackage{centernot}
\usepackage{mathtools}

\usepackage{graphics}

\usepackage{xcolor}
\usepackage{color}
\usepackage{graphicx}
\usepackage{marvosym}

\usepackage{paralist}
\usepackage{multicol}

\usepackage{pstricks,pst-text,pst-grad,pst-node,pst-3dplot,pstricks-add,pst-poly,pst-coil} 
\usepackage{pst-fun,pst-blur} 

\usepackage{hyperref}

\newtheorem{theorem}{Theorem}
\theoremstyle{plain}

\newtheorem{corollary}{Corollary}

\newtheorem{definition}{Definition}
\newtheorem{example}{Example}

\newtheorem{notation}{Notation}

\newtheorem{proposition}{Proposition}
\newtheorem{remark}{Remark}

\begin{document}

\title[Cross Ratio Advances]{Cross Ratio Geometry\\ Advances for Four Co-Linear Points in the\\ Desargues Affine Plane-Skew Field}

\author[Orgest ZAKA]{Orgest ZAKA}
\address{Orgest ZAKA: Department of Mathematics-Informatics, Faculty of Economy and Agribusiness, Agricultural University of Tirana, Tirana, Albania}
\email{ozaka@ubt.edu.al, gertizaka@yahoo.com, ozaka@risat.org}

\author[James F. Peters]{James F. Peters}
\address{James F. PETERS: Department of Electrical \& Computer Engineering, University of Manitoba, WPG, MB, R3T 5V6, Canada and Department of Mathematics, Faculty of Arts and Sciences, Ad\.{i}yaman University, 02040 Ad\.{i}yaman, Turkey}
\thanks{The research has been supported by the Natural Sciences \& Engineering Research Council of Canada (NSERC) discovery grant 185986, Instituto Nazionale di Alta Matematica (INdAM) Francesco Severi, Gruppo Nazionale  per le Strutture Algebriche, Geometriche e Loro Applicazioni grant 9 920160 000362, n.prot U 2016/000036 and Scientific and Technological Research Council of Turkey (T\"{U}B\.{I}TAK) Scientific Human Resources Development (BIDEB) under grant no: 2221-1059B211301223.}
\email{James.Peters3@umanitoba.ca}

\dedicatory{Dedicated to Girard Desargues and Karl G. C. von Staudt}


\begin{abstract}
This paper introduces advances in the geometry of the cross ratio of four co-linear points in in the Desargues affine plane. The cross-ratio of co-linear points of a skew field in the Desargues affine plane. The results given here have a clean rendition, based on Desargues affine plane axiomatics, skew field properties and the addition and multiplication of planar co-linear points. 
\end{abstract}


\maketitle
\tableofcontents

\section{Introduction and Preliminaries}

In the advancement of our research in the connections of axiomatic geometry and algebraic structures, we have achieved some results which we have presented in this paper. More recently, results are given about the association of algebraic structures in affine planes and in Desargues affine plane, and vice versa in \cite{ZakaDilauto, ZakaFilipi2016, FilipiZakaJusufi, ZakaCollineations, ZakaVertex, ZakaThesisPhd, ZakaPetersIso, ZakaPetersOrder, ZakaMohammedSF, ZakaMohammedEndo, ZakaPeters2022DyckFreeGroup}. The foundations for the study of the connections between axiomatic geometry and algebraic structures were set forth by D. Hilbert \cite{Hilbert1959geometry}. And some classic research results in this context are given, for example, by  E. Artin \cite{Artin1957GeometricAlgebra}, D.R. Huges and F.C. Piper ~\cite{HugesPiper}, H. S. M Coxeter ~\cite{CoxterIG1969}.  Marcel Berger in \cite{Berger2009geometry12}, Robin Hartshorne in \cite{Hartshorne1967Foundations}. 

In this paper, we advance in study regarding  the cross ratio of 4-points, in a line of the Desargues affine plane. We study and discuses the properties and results related to the cross ratio for four points, also we see the points of line as a elements of a skew field which constructed over this line on Desargues affine plane.

We use skew field properties for the proof of our results, since the cross-ratio sketch is very confusing (even with the Euclidean interpretation).

Earlier, we study the ratio of 2 and 3 points in a line on Desargues affine plane (see \cite{ZakaPeters2022DyckFreeGroup}, \cite{ZakaPeters2022InvariantPreserving}, \cite{ZakaPeters2022InvariantPreserving}), also we have shown that on each line on Desargues affine plane, we can construct a skew-field simply and constructively, using simple elements of elementary geometry, and only the basic axioms of Desargues affine plane (see \cite{ZakaFilipi2016}, \cite{FilipiZakaJusufi}, \cite{ZakaThesisPhd}, \cite{ZakaPetersIso} ). 

In this paper, we utilize a method that is naive and direct, without requiring the concept of coordinates. We bases only in Desargues affine plane axiomatic and in skew field properties (the points in a line on Desargues affine plane, we think of them as elements of skew fields, which is a construct over this line). 


\subsection{Desargues Affine Plane}
Let $\mathcal{P}$ be a nonempty space, $\mathcal{L}$ a nonempty subset of $\mathcal{P}$. The elements $p$ of $\mathcal{P}$ are points and an element $\ell$ of $\mathcal{L}$ is a line. 

\begin{definition}
The incidence structure $\mathcal{A}=(\mathcal{P}, \mathcal{L},\mathcal{I})$, called affine plane, where satisfies the above axioms:\\

\begin{compactenum}[1$^o$]
\item For each points $\left\{P,Q\right\}\in \mathcal{P}$, there is exactly one line $\ell\in \mathcal{L}$ such that $\left\{P,Q\right\}\in \ell$.

\item For each point $P\in \mathcal{P}, \ell\in \mathcal{L}, P \not\in \ell$, there is exactly one line $\ell'\in \mathcal{L}$ such that
$P\in \ell'$ and $\ell\cap \ell' = \emptyset$\ (Playfair Parallel Axiom~\cite{Pickert1973PlayfairAxiom}).   Put another way,
if the point $P\not\in \ell$, then there is a unique line $\ell'$ on $P$ missing $\ell$~\cite{Prazmowska2004DemoMathDesparguesAxiom}.

\item There is a 3-subset of points $\left\{P,Q,R\right\}\in \mathcal{P}$, which is not a subset of any $\ell$ in the plane.   Put another way,
there exist three non-collinear points $\mathcal{P}$~\cite{Prazmowska2004DemoMathDesparguesAxiom}.
\end{compactenum}
\end{definition}

\emph{\bf Desargues' Axiom, circa 1630}~\cite[\S 3.9, pp. 60-61] {Kryftis2015thesis}~\cite{Szmielew1981DesarguesAxiom}.   Let $A,B,C,A',B',C'\in \mathcal{P}$ and let pairwise distinct lines  $\ell^{AA_1} , \ell^{BB'}, \ell^{CC'}, \ell^{AC}, \ell^{A'C'}\in \mathcal{L}$ such that
\begin{align*}
\ell^{AA_1} \parallel \ell^{BB'} \parallel \ell^{CC'} \ \mbox{(Fig.~\ref{fig:DesarguesAxiom}(a))} &\ \mbox{\textbf{or}}\
\ell^{AA_1} \cap \ell^{BB'} \cap \ell^{CC'}=P.
 \mbox{(Fig.~\ref{fig:DesarguesAxiom}(b) )}\\
 \mbox{and}\  \ell^{AB}\parallel \ell^{A'B'}\ &\ \mbox{and}\ \ell^{BC}\parallel \ell^{B'C'}.\\
A,B\in \ell^{AB}, A'B'\in \ell^{A'B'},  &\ \mbox{and}\ B,C\in \ell^{BC},  B'C'\in \ell^{B'C'}.\\
A\neq C, A'\neq C', &\ \mbox{and}\ \ell^{AB}\neq \ell^{A'B'}, \ell^{BC}\neq \ell^{B'C'}.
\end{align*}

\begin{figure}[htbp]
	\centering
		\includegraphics[width=0.85\textwidth]{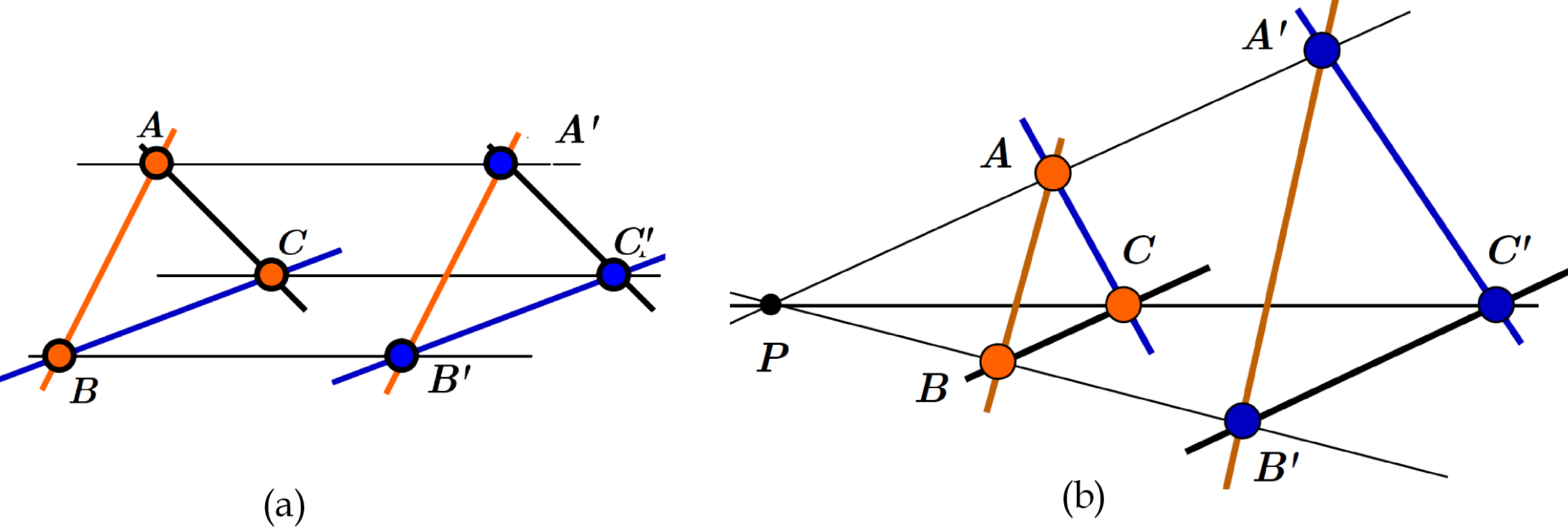}
	\caption{Desargues Axioms: (a) For parallel lines $\ell^{AA_1} \parallel \ell^{BB'} \parallel \ell^{CC'}$; (b) For lines which are cutting in a single point $P$,  $\ell^{AA_1} \cap \ell^{BB'} \cap \ell^{CC'}=P$.}
		\label{fig:DesarguesAxiom}
\end{figure}

Then $\boldsymbol{\ell^{AC}\parallel \ell^{A'C'}}$.   \qquad \textcolor{blue}{$\blacksquare$}

\begin{example}
In Euclidean plane, three vertexes $ABC$ and $A'B'C'$, are similar (in (a) are equivalent-triangle and in (b) are homothetical-triangle) the parallel lines,  $\ell^{AC}, \ell^{A'C'}\in \mathcal{L}$ in Desargues' Axiom are represented in Fig.~\ref{fig:DesarguesAxiom}.  In other words, the side $AC$ of the triangle of $\bigtriangleup ABC$ is parallel with the side $A'C'$ of the triangle $\bigtriangleup A'B'C'$, provided the restrictions on the points and lines in Desargues' Axiom are satisfied.
\qquad \textcolor{blue}{$\blacksquare$}
\end{example}


\noindent A {\bf Desargues affine plane} is an affine plane that satisfies Desargues' Axiom. 
\begin{notation}
Three vertexes $ABC$ and $A'B'C'$, which, fulfilling the conditions of the Desargues Axiom, we call \emph{'Desarguesian'}.
\end{notation}

\subsection{Addition and Multiplication of points in a line of Desargues affine plane} $ $\\

\textbf{Addition of points in a line of affine plane:} In an Desargues affine plane $\mathcal{A_D}=(\mathcal{P},\mathcal{L},\mathcal{I})$ we fix two different points $O,I\in \mathcal{P},$ which, according to Axiom 1, determine a line $\ell^{OI}\in \mathcal{L}.$ Let $A$ and $B$ be two arbitrary points of a line $\ell^{OI}$. In plane $\mathcal{A_D}$ we choose a point $B_{1}$ not incident with $\ell^{OI}$: $B_{1}\notin \ell^{OI}$ (we call the auxiliary point). Construct line $\ell_{OI}^{B_{1}},$ which is only according to the Axiom 2. Then construct line $\ell_{OB_{1}}^{A},$ which also is the only according to the Axiom 2. Marking their intersection $P_{1}=\ell_{OI}^{B_{1}}\cap \ell_{OB_{1}}^{A}.$ Finally construct line $\ell_{BB_{1}}^{P_{1}}.$ For as much as $\ell^{BB_{1}}$ cuts the line $\ell^{OI}$ in point $B$, then this line, parallel with $\ell^{BB_{1}}$, cuts the line $\ell^{OI}$ in a single point $C$, this point we called the addition of points $A$ with point $B$ (Figure \ref{fig:FigureAdMult} (a)).

\textbf{Multiplication of points in a line in affine plane}.
Choose in the plane $\mathcal{A_D}$ one point $B_{1}$ not incident with lines $\ell^{OI},$ and construct the line $\ell^{IB_{1}}$. Construct the line $\ell_{IB_{1}}^{A},$ which is the only accoding to the Axiom 2 and cutting the line $\ell^{OB_{1}}$. Marking their intersection with $P_{1}=\ell
_{IB_{1}}^{A}\cap OB_{1}.$ Finally, construct the line $\ell
_{BB_{1}}^{P_{1}}.$ Since $\ell^{BB_{1}}$ cuts the line $\ell^{OI}$ in a single point $B$, then this line, parallel with $\ell^{BB_{1}}$, cuts the line $\ell^{OI}$ in one single point $C$, this point we called the multiplication of points $A$ with point $B$ (Figure \ref{fig:FigureAdMult} (b)).

The process of construct the points $C$ for adition and multiplication of points in $\ell^{OI}-$line in affine plane, is presented in the tow algorithm form  

\begin{multicols}{2}
\textsc{Addition Algorithm}
\begin{description}
	\item[Step.1] $B_{1}\notin \ell^{OI}$
	\item[Step.2] $\ell_{OI}^{B_{1}}\cap \ell_{OB_{1}}^{A}=P_{1}$
	\item[Step.3] $\ell_{BB_{1}}^{P_{1}}\cap \ell^{OI}=C(=A+B)$
\end{description}

\textsc{Multiplication Algorithm}
\begin{description}
	\item[Step.1] $B_{1}\notin \ell^{OI}$
	\item[Step.2] $\ell_{IB_{1}}^{A}\cap \ell^{OB_{1}}=P_{1}$
	\item[Step.3] $\ell_{BB_{1}}^{P_{1}}\cap \ell^{OI}=C(=A\cdot B)$
\end{description}
\end{multicols}

\begin{figure}[htbp]
\centering%
\includegraphics[width=0.92\textwidth]{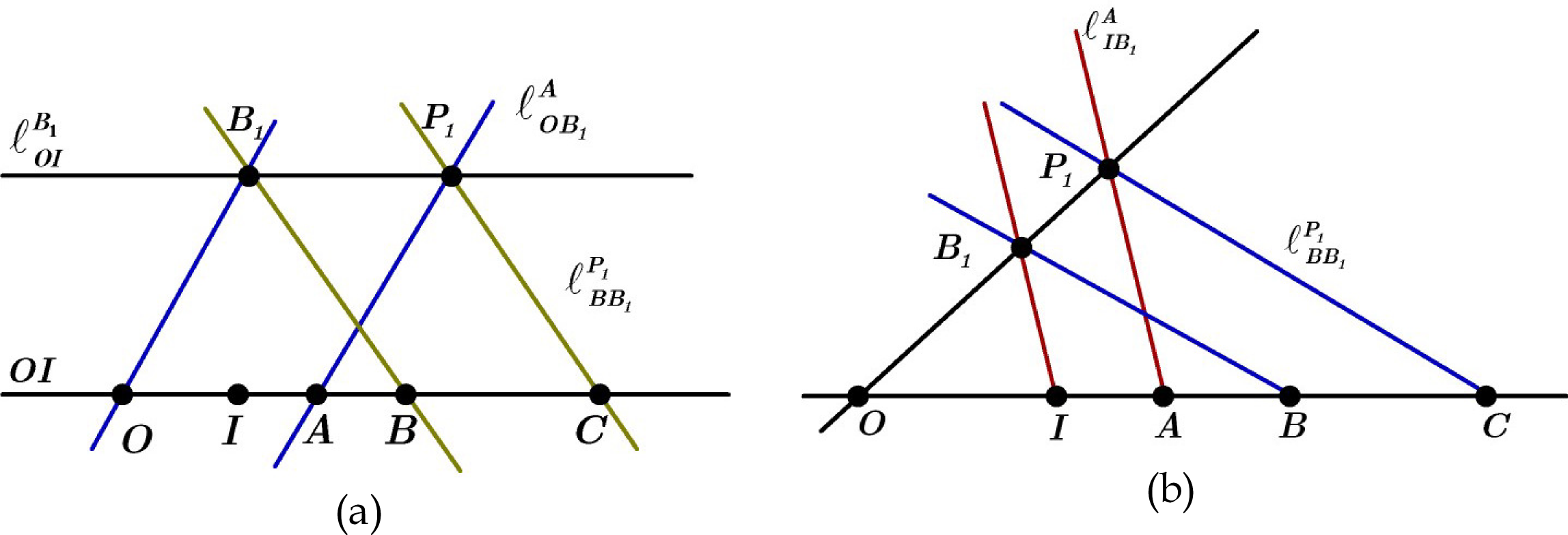}
\caption{ (a) Addition of points in a line in affine plane, 
(b) Multiplication of points in a line in affine plane}
\label{fig:FigureAdMult}
\end{figure}

In \cite{ZakaThesisPhd} and \cite{FilipiZakaJusufi}, we have prove that $(\ell^{OI}, +, \cdot)$ is a skew field in Desargues affine plane, and is field (commutative skew field) in the Papus affine plane.

\subsection{Some algebraic properties of Skew Fields}

I n this section $K$ will denote a skew field~\cite{Herstein1968NR} and $z[K]$ its center, where is the set $K$ such that
\[
z[K]=\left\{k \in K \quad |\quad ak=ka, \quad \forall a \in K \right\}
\]
\begin{proposition}
$z[K]$ is a commutative subfield of a skew field $K$.
\end{proposition}
%
Let now $p\in K$ be a fixed element of the skew field $K$. We will denote $z_K(p)$ the centralizer in $K$ of the element $p$, where is the set,
\[
z_K(p)=\left\{k \in K | pk=kp, \right\}.
\]
$z_K(p)$ is sub skew field of K, but, in general, it is not commutative.

Let $K$ be a skew field, $p\in K$, and let us denote by $[p_K]$ the conjugacy class of $p$:
\[
[p_K]= \left\{q^{-1}pq \quad|\quad q \in K \setminus \{0\} \right\}
\]
If, $p\in z[K]$, for all $q \in K$ we have that $q^{-1}pq=p.$

\subsection{Ratio of two and three points}
In the paper \cite{ZakaPeters2022DyckFreeGroup}, we have done a detailed study, related to the ratio of two and three points in a line of Desargues affine plane. Below we are listing some of the results for ratio of two and three points.
  
\begin{definition} \label{ratio2points}
\cite{ZakaPeters2022DyckFreeGroup} Lets have two different points $A,B \in \ell^{OI}-$line, and $B\neq O$, in Desargues affine plane. We define as ratio of this tow points, a point $R\in \ell^{OI}$, such that,
\[R=B^{-1}A, \qquad \text{
we mark this, with,} \qquad 
R=r(A:B)=B^{-1}A
\]
\end{definition}

For a 'ratio-point' $R \in \ell^{OI}$, and for point $B\neq O$ in line $\ell^{OI}$, is a unique defined point, $A \in \ell^{OI}$, such that $R=B^{-1}A=r(A:B)$.
 
\begin{figure}[htbp]
	\centering
		\includegraphics[width=0.75\textwidth]{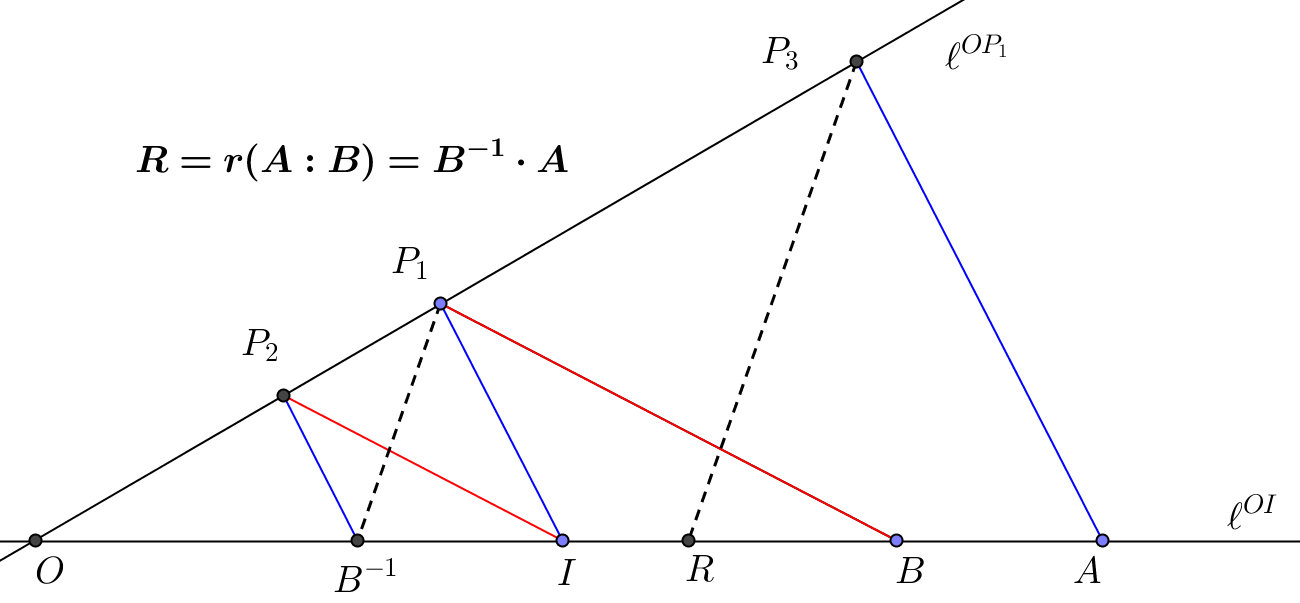}
	\caption{Ilustrate the Ratio-Point, of 2-Points in a line of Desargues affine plane $R=r(A:B)=B^{-1}A$.}
	\label{Ratio2points}
\end{figure}

\textbf{Some results for Ratio of 2-points in Desargues affine plane} (see \cite{ZakaPeters2022DyckFreeGroup}).
\begin{itemize}
\item If have two different points $A,B \in \ell^{OI}-$line, and $B\neq O$, in Desargues affine plane, then, $
r^{-1}(A:B)=r(B:A)$. 
%
\item For three collinear point $A,B,C$ and $C\neq O$, in $\ell^{OI}-$line, have, 
\[
r(A+B:C)=r(A:C)+r(B:C).
\]
\item For three collinear point $A,B,C$ and $C\neq O$, in $\ell^{OI}-$line, have,
\begin{enumerate}
	\item $r(A\cdot B:C)=r(A:C)\cdot B.$
	\item $r(A:B\cdot C)=C^{-1}r(A:C).$
\end{enumerate}
%
\item Let's have the points $A,B \in \ell^{OI}-$line where $B\neq O$.  Then have that, 
\[
r(A:B)=r(B:A) \Leftrightarrow A=B.
\]
%
\item This ratio-map, $r_{B}: \ell^{OI} \to \ell^{OI}$ is a bijection in $\ell^{OI}-$line in Desargues affine plane. 
%
\item The ratio-maps-set $\mathcal{R}_2=\{r_{B}(X)|\forall X\in \ell^{OI} \}$, for a fixed point $B$ in $\ell^{OI}-$line, forms a skew-field with 'addition and multiplication' of points. 
This, skew field $(\mathcal{R}_2, +, \cdot)$ is sub-skew field of the skew field $(\ell^{OI}, +, \cdot)$.
\end{itemize}
\textbf{Ratio of three points in a line on Desargues affine plane.} (see \cite{ZakaPeters2022DyckFreeGroup})
\begin{definition}\label{ratiodef}
If $A, B, C$ are three points on a line $\ell^{OI}$ (collinear) in Desargues affine plane, then we define their \textbf{ratio} to be a point $R \in \ell^{OI}$, such that:
\[
(B-C)\cdot R=A-C, \quad \mbox{concisely}\quad R=(B-C)^{-1}(A-C),
\]
and we mark this with  $r(A,B;C)= (B-C)^{-1}(A-C)$.
\end{definition}

\begin{figure}[htbp]
	\centering
		\includegraphics[width=0.9\textwidth]{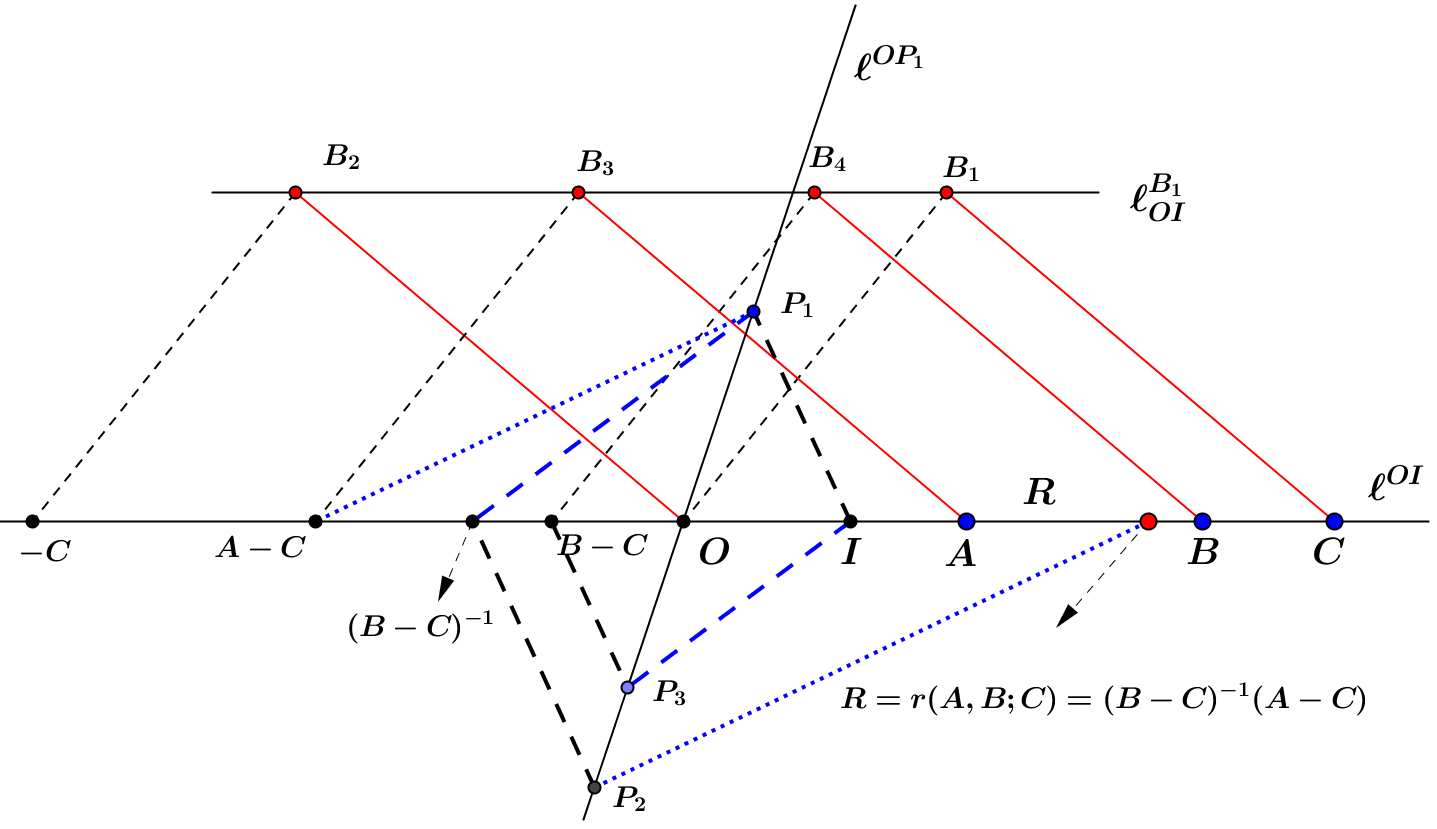}
	\caption{Ratio of 3-Points in a line of Desargues affine plane $R=r(A,B;C)$.}
		\label{ratio3points1}
\end{figure}

\textbf{Some Results for Ratio of 3-points in Desargues affine plane} (\cite{ZakaPeters2022DyckFreeGroup}).
\begin{itemize}
\item \label{reverse.ratio} For 3-points $A,B,C$ in a line $\ell^{OI}$ of Desargues affine plane, we have that,
\[
r(-A,-B;-C)=r(A,B;C).
\]
%
\item \label{inversratio} For 3-points $A,B,C$ in a line $\ell^{OI}$ in the Desargues affine plane, have
\[r^{-1}(A,B;C)=r(B,A;C).
\]
\item If  $A,B,C$, are three different points, and different from point $O$, in a line $\ell^{OI}$ on Desargues affine plane, then
\[r(A^{-1},B^{-1};C^{-1})=B[r(A,B;C)]A^{-1}.\]
%
\item In the Pappus affine plane, for three point different from point $O$, in $\ell^{OI}-$line, we have
$r(A^{-1},B^{-1};C^{-1})=r(A,B;C) \cdot r(B,A;O).$
%
%
\item This ratio-map, $r_{BC}: \ell^{OI} \to \ell^{OI}$ is a bijection in $\ell^{OI}-$line in Desargues affine plane. 
%
%
%
%
\item The ratio-maps-set $\mathcal{R}_3=\{r_{BC}(X)|\forall X\in \ell^{OI} \}$, for a different fixed points $B,C$ in $\ell^{OI}-$line, forms a skew-field with 'addition and multiplication' of points in $\ell^{OI}-$line.
This, skew field $(\mathcal{R}_3, +, \cdot)$ is sub-skew field of the skew field $(\ell^{OI}, +, \cdot)$.
\end{itemize}

\section{Cross-Ratio for Fuor points in a line of Desargues affine plane}

This section culminates in a main result in this paper. We  consider the cross-ratio of co-linear points in Desargues affine planes, utilizing a method that is naive and direct without requiring planar coordinates.  We define the cross-ratio of four co-linear points in a line on Desargues affine plane \emph{as a point in this line}. This work carries forward earlier results that reveal the close connection between lines in the Desargues affine planes and corresponding skew fields.  Skew fields properties in our proofs. Mainly, we rely on our results regarding the addition and multiplication of co-linear points in the Desargues affine plane, and the fact that a line (set of points), with addition and multiplication, forms a skew field (for more about this, see \cite{ZakaThesisPhd}, \cite{ZakaFilipi2016}, \cite{FilipiZakaJusufi}, \cite{ZakaVertex}, \cite{ZakaCollineations}, \cite{ZakaDilauto},  \cite{ZakaPetersIso}, \cite{ZakaPetersOrder}, \cite{ZakaPeters2022InvariantPreserving}).

The classical definition of the cross-ratio (see \cite{Milne1911elementaryCross-Ratio, Hilbert1959geometry, Berger2010geometryRevealed, Berger2009geometry12}) for 4-points, is given as a product of tow ratio of lengths.  So, for example, for four co-linear points $A,B,C,D$,
\[
c_r(A,B;C,D)=\frac{AC}{BC} \cdot \frac{BD}{AD},
\]
where $AC,BC,BD,AD$ are the lengths of segments $[AB],[BC],[BD].[AD]$, respectively. 

Since we will not use coordinates and metrics, our definitions are rely solely on the algebra and axiomatics for the Desargues affine plane.

Let us have the line $\ell^{OI}$ in Desargues affine plane $\mathcal{A_{D}}$, and four points, $A, B, C, D \in \ell^{OI}$

\begin{definition}\label{cross-ratio.def}
If $A, B, C, D$ are four points on a line $\ell^{OI}$ in Desarges affine plane $\mathcal{A_{D}}$, no three of them equal, then we define their cross ratio to be a point:
\[c_r(A,B;C,D)=\left[(A-D)^{-1}(B-D)\right]\left[(B-C)^{-1}(A-C)\right]
\]
\end{definition}

\begin{remark}
Similar to 'ratio', we can define it, the cross-ratio, also as
\[ c_r(A,B;C,D)=[(B-D)(A-D)^{-1}][(A-C)(B-C)^{-1}], \]
or
\[ c_r(A,B;C,D)=[(B-D)(A-C)][(A-D)^{-1}(B-C)^{-1}], \]
(or all combination of product of this 4-factors)
the results would be similar, \emph{but the obtained point will always be different for each case}. In $\ell^{OI}-$line, in Desargues affine planes, these are a different point from that of our definition, since:
\[ [(B-D)(A-D)^{-1}][(A-C)(B-C)^{-1}] \neq \left[(A-D)^{-1}(B-D)\right]\left[(B-C)^{-1}(A-C)\right].
\]
and
\[ [(B-D)(A-C)][(A-D)^{-1}(B-C)^{-1}] \neq \left[(A-D)^{-1}(B-D)\right]\left[(B-C)^{-1}(A-C)\right].
\]
also for the other cases, we would have a difference for each pair, found for the cross ratio, according to any definition we take.
We are keeping our definition.
\end{remark}

\begin{definition}
If the line $\ell^{OI}$ in Desargues affine plane, is a infinite line (number of points in this line is $+\infty$), we define as follows:
\begin{equation*}
\begin{aligned}
	c_r(\infty, B;C,D) &=(B-D)(B-C)^{-1}\\
	c_r(A,\infty;C,D) &= (A-D)^{-1}(A-C)\\
 c_r(A,B;\infty, D)&=(A-D)^{-1}(B-D) \\
c_r(A,B;C,\infty)&=(B-C)^{-1}(A-C) 
\end{aligned}
\end{equation*}
\end{definition}

From this definition and from ratio definition \ref{ratiodef} we have that,
\begin{itemize}
	\item $c_r(A,B;C,D)=\left[(A-D)^{-1}(B-D)\right]\left[(B-C)^{-1}(A-C)\right]$, so
	\[ c_r(A,B;C,D)=r(B,A;D) \cdot r(A,B;C).\]
	\item $c_r(\infty, B;C,D)=(B-D)(B-C)^{-1}=[(D-B)^{-1}(C-B)]^{-1}$, so,
	\[c_r(\infty, B;C,D)=r^{-1}(C,D;B).\]
	\item $c_r(A,\infty;C,D)= (A-D)^{-1}(A-C)=(D-A)^{-1}(C-A)$, so, 
	\[c_r(A,\infty;C,D)=r(C,D;A).\]
	\item $c_r(A,B;\infty, D)=(A-D)^{-1}(B-D)$, so
	\[c_r(A,B;\infty, D)=r(A,B;D).\]
	\item $c_r(A,B;C,\infty)=(B-C)^{-1}(A-C)$, so,
	\[c_r(A,B;C,\infty)=r(A,B;C).\]
\end{itemize}

\textbf{Some simple properties of Cross-Ratios}, which derive directly from the definition, related to the position of the points $A,B,C,D$ in $\ell^{OI}-$line in Desargues affine plane.

\begin{itemize}
	\item If $A=B$, then 
	\[ 
	\begin{aligned}
	c_r(A,B;C,D)&=c_r(A,A;C,D)\\
	&=[(A-D)^{-1}(A-D)][(A-C)^{-1}(A-C)] \\
	&=[I][I]\\
	&=I.
	\end{aligned}
	\]
	\item If $A=C$, then 
	\[
	\begin{aligned}
	c_r(A,B;C,D)&=c_r(A,B;A,D)\\
	&=[(A-D)^{-1}(B-D)][(B-A)^{-1}(A-A)] \\
	&=[(A-D)^{-1}(B-D)][(B-A)^{-1}\cdot O]\\
	&=O.
	\end{aligned}
	\]
	\item If $A=D$, then 
	\[ 
	\begin{aligned}
	c_r(A,B;C,D)&=c_r(A,B;C,A)\\
	&=[(A-A)^{-1}(B-A)][(B-C)^{-1}(A-C)] \\
	&=[O^{-1}(B-A)][(B-C)^{-1}(A-C)] \\
	&\text{(think that $O^{-1}=\infty$(point in infinity))}\\
	&=\infty.
	\end{aligned}
	\]
	\item If $B=C$, then 
	\[
	\begin{aligned}
	c_r(A,B;C,D)&=c_r(A,B;B,D)\\
	&=[(A-D)^{-1}(B-D)][(B-B)^{-1}(A-B)] \\
	&=[(A-D)^{-1}(B-D)][O^{-1}(A-B)] \\
	&\text{(think that $O^{-1}=\infty$(point in infinity))}\\
	&=\infty. 
	\end{aligned}
	\]
	\item If $B=D$, then 
	\[
	\begin{aligned}
	c_r(A,B;C,D)&=c_r(A,B;C,B)\\
	&=[(A-B)^{-1}(B-B)][(B-C)^{-1}(A-C)] \\
	&=[(A-B)^{-1}\cdot O][(B-C)^{-1}(A-C)] \\
	&=O. 
	\end{aligned}
	\]
	\item If $C=D$, then 
	\[ 
	\begin{aligned}
	c_r(A,B;C,D)&=c_r(A,B;C,C) \\
	&=[(A-C)^{-1}(B-C)][(B-C)^{-1}(A-C)] \\
	&=(A-C)^{-1}[(B-C)(B-C)^{-1}](A-C) \\
	&=(A-C)^{-1}\cdot I \cdot (A-C)\\
	&= (A-C)^{-1}(A-C) \\
	&=I.
	\end{aligned}
	\]
\end{itemize}

\begin{theorem}
Let $R\in \ell^{OI}$, such that $R\neq O$ and $R \neq I$. If $A, B, C \in \ell^{OI}$ are three different points, then exist a single point $D \in \ell^{OI}$, such that $c_r(A,B;C,D)=R.$
\end{theorem}
\proof
Suppose that exist tow different points $D$ an $D'$ in $\ell^{OI}-$line, such that 
\[c_r(A,B;C,D)=c_r(A,B;C,D')\]
We rewrite them, cross ratios, as products of 'ratios', and we have,
\[
c_r(A,B;C,D)=\left[(A-D)^{-1}(B-D)\right]\left[(B-C)^{-1}(A-C)\right]
=r(B,A;D)\cdot r(A,B;C)
\]
and
\[
c_r(A,B;C,D')=\left[(A-D')^{-1}(B-D')\right]\left[(B-C)^{-1}(A-C)\right]
=r(B,A;D')\cdot r(A,B;C)
\]
So, have,

\[
r(B,A;D)\cdot r(A,B;C)=r(B,A;D')\cdot r(A,B;C)
\]
we mark $r(B,A;D)=R_1;  r(A,B;C)=R_2, r(B,A;D')=R_3$, 
remember that these are points of the line $\ell^{OI}$, 
so they are elements of the skew-fields $(\ell^{OI},+, \cdot)$, and have
\[R=R_1 \cdot R_2 \quad \text{and}\quad R= R_3 \cdot R_2 \]
Thus, for it, we have
\[
R_1 \cdot R_2 = R_3 \cdot R_2 \Rightarrow R_1 \cdot R_2 - R_3 \cdot R_2=O \Rightarrow (R_1 - R_3) \cdot R_2=O
\]

But the points, $R_1,R_2,R_3$, are points of $\ell^{OI}-$line in Desargues affine plane, therefore, they are elements of skew-fields $K=(\ell^{OI},+, \cdot)$. We also know the fact that 'a skew field does not have a divisor of zero' (more on skew fields, see \cite{Cohn2008sf}, \cite{Herstein1968NR}, \cite{Rotman2015AMAlgebra}, \cite{Lam2001GTMalgebra})
\[
 R_1 - R_3=O \quad \text{or}\quad R_2=O, \quad \text{but}\quad R_2\neq O \Rightarrow R_1 - R_3=O
\]
so,
\[
R_1=R_3 \Rightarrow r(B,A;D)=r(B,A;D')
\]
and from the uniqueness of the definition for 'ratio', we have, 
\[D=D'\]

\qed

\begin{theorem}\label{cross-ratio.reverse}
\bigskip 
If $A,B,C,D$ are distinct points in a $\ell^{OI}-$line, in Desargues affine plane, then
\[ c_r(-A,-B;-C,-D)=c_r(A,B;D,C)
\]
\end{theorem}
\proof
From cross-ratio definition \ref{cross-ratio.def}, we have

\begin{equation*}
\begin{aligned}
	c_r(-A,-B;-C,-D) & =[(-A-(-C))^{-1}(-B-(-C))][(-B-(-D))^{-1}(-A-(-D))]\\
	& =[(-A+C)^{-1}(-B+C)][(-B+D)^{-1}(-A+D)]\\
	&= [([-I](A-C))^{-1}[-I](B-C)][([-I](B-D))^{-1}[-I](A-D)] \\
	&= [(A-C)^{-1}[-I]^{-1}[-I](B-C)][(B-D)^{-1}[-I]^{-1}[-I](A-D)] \\
	&= [(A-C)^{-1}[-I][-I](B-C)][(B-D)^{-1}[-I][-I](A-D)] \\
	&= [(A-C)^{-1}(B-C)][(B-D)^{-1}(A-D)] \\
	&=c_r(A,B;D,C)
\end{aligned}
\end{equation*}
From skew fields properties we have that $(ab)^{-1}=b^{-1}a^{-1}$and $ab\neq ba$,  $[-I]^{-1}=-I$, and $[-I][-I]=I.$
\qed

\begin{theorem}\label{cross-ratio.invers}
\bigskip 
If $A,B,C,D$ are distinct points in a line, in Desargues affine plane, then
\[ c_r^{-1}(A,B;C,D)=c_r(A,B;D,C)
\]
\end{theorem}
\proof
From cross-ratio Definition \ref{cross-ratio.def}, have
\begin{equation*}
\begin{aligned}
	c_r^{-1}(A,B;C,D) & =\left\{\left[(A-D)^{-1}(B-D)\right]\left[(B-C)^{-1}(A-C)\right]\right\}^{-1}\\
	&=\left[(B-C)^{-1}(A-C)\right]^{-1}  \left[(A-D)^{-1}(B-D)\right]^{-1}\\
	&= [(A-C)^{-1}(B-C)][(B-D)^{-1}(A-D)] \\
	&=c_r(A,B;D,C)
\end{aligned}
\end{equation*}
\qed

\begin{theorem}
For 4 co-linear points $A,B,C,D$ in a line $\ell^{OI}$ in the Desargues affine plane, the cross-ratio satisfies the equation,
\[
c_r(A,B;C,D)=\left[(A-B)^{-1}-(A-D)^{-1}\right] \left[(A-B)^{-1}-(A-C)^{-1}\right]^{-1}
\]
\end{theorem}
\begin{proof}
From the definition \ref{cross-ratio.def} we have $c_r(A,B;C,D)=\left[(A-D)^{-1}(B-D)\right]\left[(B-C)^{-1}(A-C)\right]$, 
and the points of this line forms a skew-field, therefore, we have association property:
\[
\begin{aligned}
R&=c_r(A,B;C,D)\\
&=[(A-D)^{-1}(B-D)][(B-C)^{-1}(A-C)]\\
&\text{(since this factor are elements of skew field, so have the associative property)}\\
&=\left[(A-D)^{-1}(B-D)(B-C)^{-1}\right](A-C)
\end{aligned}
\]
So, the point $R$, is,
\[
R=\left[(A-D)^{-1}(B-D)(B-C)^{-1}\right](A-C)
\]
multiply in the right side by side with $(A-C)^{-1}$, and have
\[
R\cdot (A-C)^{-1}=(A-D)^{-1}(B-D)(B-C)^{-1}
\]
now multiply in the right side by side with $(B-C)$, and have
\[(A-D)^{-1}(B-D)=\left(R \cdot (A-C)^{-1}\right)(B-C)
\]
multiply side by side with $(A-B)^{-1}$, and have,
\[
\left[(A-D)^{-1}(B-D)\right](A-B)^{-1}=\left[R \cdot (A-C)^{-1}(B-C)\right](A-B)^{-1}
\]
Transform the left side of this equation as
\[ \begin{aligned}
\left[(A-D)^{-1}(B-D)\right](A-B)^{-1}&=\left[(A-D)^{-1}(B+A-A-D)\right](A-B)^{-1}\\
&=\left[(A-D)^{-1}([A-D]-[A-B])\right](A-B)^{-1}
\end{aligned} \]
rewrite it, and have
\begin{equation*}
\begin{aligned}
	\left[(A-D)^{-1}(B-D)\right](A-B)^{-1}&=[(A-D)^{-1}([A-D]-[A-B])](A-B)^{-1}\\
	&= [(A-D)^{-1}[A-D]-(A-D)^{-1}[A-B]](A-B)^{-1}\\
 &=[I-(A-D)^{-1}[A-B]](A-B)^{-1} \\
 &=  (A-B)^{-1}-(A-D)^{-1}[A-B](A-B)^{-1} \\
&= (A-B)^{-1}-(A-D)^{-1}
\end{aligned}
\end{equation*}
So, have 
\[
(A-B)^{-1}-(A-D)^{-1}=\left[R(A-C)^{-1}(B-C)\right](A-B)^{-1}
\]
In the same way as above, (always bearing in mind that the points of a line of Desargues affine planes form a skew-field related to the addition and multiplication of the points, on this line, and the properties that satisfy a skew-field) we do the following transformations.

First we have the associative property for the multiplication of points on a line,
\[
\left[R(A-C)^{-1}(B-C)\right](A-B)^{-1}=R\left[(A-C)^{-1}(B-C)(A-B)^{-1}\right]
\]
Now we transform the expression

\begin{equation*}
\begin{aligned}
	\left[(A-C)^{-1}(B-C)(A-B)^{-1}\right]
	&= \left[(A-C)^{-1}(B+A-A-C)(A-B)^{-1}\right]\\
	&= \left[(A-C)^{-1}([A-C]-[A-B])(A-B)^{-1}\right]\\
	&=[I-(A-C)^{-1}[A-B]](A-B)^{-1} \\
	&=  (A-B)^{-1}-(A-C)^{-1}[A-B](A-B)^{-1} \\
&= (A-B)^{-1}-(A-C)^{-1}
\end{aligned}
\end{equation*}
So, have 
\[
(A-B)^{-1}-(A-D)^{-1}=R \left[(A-B)^{-1}-(A-C)^{-1}\right]
\]

Hence
\[
R=\left[(A-B)^{-1}-(A-D)^{-1}\right] \left[(A-B)^{-1}-(A-C)^{-1}\right]^{-1}
\]
so,
\[
c_r(A,B;C,D)=\left[(A-B)^{-1}-(A-D)^{-1}\right] \left[(A-B)^{-1}-(A-C)^{-1}\right]^{-1}.
\]

\end{proof}

\begin{theorem}\label{identity.theorem}
\bigskip 
If $A,B,C,D$ are distinct points in a line, in Desargues affine plane and $I$ is unital point for multiplications of points in same line, then
\[ I-c_r(A,B;C,D)=c_r(A,C;B,D)
\]
\end{theorem}
\proof
Let's start the calculations, using the result of the theorem
\[
I-\left[(A-B)^{-1}-(A-D)^{-1}\right] \left[(A-B)^{-1}-(A-C)^{-1}\right]^{-1} =\]
\[= \left\{(A-C)^{-1}-(A-D)^{-1}\right\}
\left[\left\{(A-C)^{-1}-(A-B)^{-1}\right\}\right]^{-1}
\]
write,
\[
I=\left[(A-B)^{-1}-(A-C)^{-1}\right]\left[(A-B)^{-1}-(A-C)^{-1}\right]^{-1}
\]
so,
\begin{equation*}
\begin{aligned}
I-c_r(A,B;C,D) &=I-\left[(A-B)^{-1}-(A-D)^{-1}\right] \left[(A-B)^{-1}-(A-C)^{-1}\right]^{-1} \\
&= \left[(A-B)^{-1}-(A-C)^{-1}\right]\left[(A-B)^{-1}-(A-C)^{-1}\right]^{-1}\\
&-\left[(A-B)^{-1}-(A-D)^{-1}\right] \left[(A-B)^{-1}-(A-C)^{-1}\right]^{-1}\\
& =
\left\{\left[(A-B)^{-1}-(A-C)^{-1}\right]-\left[(A-B)^{-1}-(A-D)^{-1}\right]\right\} \\
& \cdot \left[(A-B)^{-1}-(A-C)^{-1}\right]^{-1}\\
&= \left\{-(A-C)^{-1}+(A-D)^{-1}\right\}\left[(A-B)^{-1}-(A-C)^{-1}\right]^{-1}\\
&=(-I)\left\{(A-C)^{-1}-(A-D)^{-1}\right\}
\left[(-I)\left\{(A-C)^{-1}-(A-B)^{-1}\right\}\right]^{-1}\\
&= (-I)\left\{(A-C)^{-1}-(A-D)^{-1}\right\}
\left[\left\{(A-C)^{-1}-(A-B)^{-1}\right\}\right]^{-1}(-I)^{-1}\\
&= \left\{(A-C)^{-1}-(A-D)^{-1}\right\}
\left[\left\{(A-C)^{-1}-(A-B)^{-1}\right\}\right]^{-1}\\
&=c_r(A,C;B,D)
\end{aligned}
\end{equation*}

from skew-field properties, we have $(-I)^{-1}=-I$ and $(-I)(-I)=I$
\qed

\begin{theorem}
If $A,B,C,D$ are distinct points in a line, in Desargues affine plane and $I$ is unitary point for multiplications of points in same line, then,
\begin{description}
	\item[(a)] $c_r(A,D;B,C)=I-c_r^{-1}(A,B;C,D)$
	\item[(b)] $c_r(A,C;D,B)=[I-c_r(A,B;C,D)]^{-1}$
	\item[(c)] $c_r(A,D;C,B)=[c_r(A,B;C,D)-I]^{-1}c_r(A,B;C,D)$
\end{description}
\end{theorem}
\proof
(a) In theorem \ref{cross-ratio.invers} we have prove that $c_r^{-1}(A,B;C,D)=c_r(A,B;D,C)$, and from theorem \ref{identity.theorem}, have that $I-c_r(A,B;D,C)=c_r(A,D;B,C)$. So, we have prove that 
\[I-c_r^{-1}(A,B;C,D)=I-c_r(A,B;D,C)=c_r(A,D;B,C). \]

(b) From theorem \ref{identity.theorem}, we have that, $I-c_r(A,B;C,D)=c_r(A,C;B,D)$, and from theorem \ref{cross-ratio.invers} have that $[c_r(A,C;B,D)]^{-1}=c_r(A,C;D,B)$, so have that 
\[ c_r(A,C;D,B)= [c_r(A,C;B,D)]^{-1} =[I-c_r(A,B;C,D)]^{-1}.\]

(c) At this point we will prove that: $c_r(A,D;C,B)=[c_r(A,B;C,D)-I]^{-1}c_r(A,B;C,D)$.
 
From point (a), we prove that $c_r(A,D;C,B)=I-c_r^{-1}(A,B;C,D),$ and from theorem \ref{cross-ratio.invers} have that $c_r(A,D;C,B)=c_r^{-1}(A,D;B,C)$. So, we have that 
\[c_r(A,D;C,B)=[I-c_r^{-1}(A,B;C,D)]^{-1}
\]
Mark the cross-ratios point $R=c_r(A,B;C,D)$, and rewrite. So we have to prove that the equation holds,
\[
[I-R^{-1}]^{-1}=[R-I]^{-1}R,
\] 
remember that the points are points of $\ell^{OI}-$line, in Desargues affine planes, and can also be thought of as elements of skew-fields $K=(\ell^{OI},+,\cdot)$, therefore, we can make algebraic transformations, allowed for skew-fields, and we have

\[
\begin{aligned}
\left[ I-R^{-1} \right]^{-1}&=[R-I]^{-1}R\\
\text{(multiply from the }&\text{right with $R^{-1}$)}\\
[I-R^{-1}]^{-1}\cdot R^{-1}&=[R-I]^{-1}R \cdot R^{-1}  \\
\text{(from skew field property }&\text{ have that $p^{-1}q^{-1}=(qp)^{-1}$)}\\
[R(I-R^{-1})]^{-1}&=[R-I]^{-1} [R \cdot R^{-1}]\\
[R\cdot I- R\cdot R^{-1}]^{-1} &=[R-I]^{-1} \cdot I\\
[R- I]^{-1} &=[R-I]^{-1}
\end{aligned}
\]
\qed

\begin{theorem}
If $A,B,C,D$ are distinct points, and different from zero-point $O$, in a line, in Desargues affine plane and $I$ is unitary point for multiplications of points in same line, have,
\[
c_r(A^{-1},B^{-1}; C^{-1},D^{-1}) = A\cdot c_r(A,B;C,D) \cdot A^{-1}
\]
\end{theorem}
\proof
From cross-ratio definition \ref{cross-ratio.def}, we have,
\[
c_r(A^{-1},B^{-1}; C^{-1},D^{-1}) =
[(A^{-1}-D^{-1})^{-1}(B^{-1}-D^{-1})][(B^{-1}-C^{-1})(A^{-1}-C^{-1})]
\]
Points $A,B,C,D$ and $A^{-1},B^{-1}, C^{-1},D^{-1}$, are points of $\ell^{OI}-$line in Desargues affine plane, so are and elements of the skew field $K=(\ell^{OI},+,\cdot)$. First we prove that, for tow elements $X,Y$ in a skew field $K$, we have that $X^{-1}-Y^{-1}=Y^{-1}(Y-X)X^{-1}$. Indeed
	\[
	\begin{aligned}
	Y^{-1}(Y-X)X^{-1}&=[Y^{-1}(Y-X)]X^{-1}\\
	&=(Y^{-1}Y-Y^{-1}X)X^{-1}\\
	&=(I-Y^{-1}X)X^{-1}\\
	&=IX^{-1}-Y^{-1}(XX^{-1})\\
	&=X^{-1}-Y^{-1}I\\
	&=X^{-1}-Y^{-1}. 
	\end{aligned}
	\]
We use this result in the calculation of $c_r(A^{-1},B^{-1}; C^{-1},D^{-1})$, and have

\[
\begin{aligned}
c_r(A^{-1},B^{-1}; C^{-1},D^{-1}) &=[(A^{-1}-D^{-1})^{-1}(B^{-1}-D^{-1})] \\
& \cdot [(B^{-1}-C^{-1})(A^{-1}-C^{-1})] \\
&= [(D^{-1}(D-A)A^{-1})^{-1}(D^{-1}(D-B)B^{-1})] \\
& \cdot [(C^{-1}(C-B)B^{-1})(C^{-1}(C-A)A^{-1})] \\
& =[(A(D-A)^{-1}D)( D^{-1}(D-B)B^{-1})] \\
&\cdot [(B(C-B)^{-1}C)(C^{-1}(C-A)A^{-1})] \\
&\text{(from skew field properties $(abc)^{-1}=c^{-1}b^{-1}a^{-1}$)}\\
& =[A(D-A)^{-1}(D D^{-1})(D-B)B^{-1}] \\
& \cdot [B(C-B)^{-1}(C C^{-1})(C-A)A^{-1}]\\
&\text{(from associative properties for multiplication)}\\
&=[A(D-A)^{-1}(I)(D-B)B^{-1}]
[B(C-B)^{-1}(I)(C-A)A^{-1}]
\\
&= [A(D-A)^{-1}(D-B)B^{-1}]
[B(C-B)^{-1}(C-A)A^{-1}]
\\
&= A[(D-A)^{-1}(D-B)B^{-1}]
[B(C-B)^{-1}(C-A)]A^{-1}
\\
&= A\left\{[(D-A)^{-1}(D-B)B^{-1}]
[B(C-B)^{-1}(C-A)]\right\}A^{-1}
\\
&=A\cdot c_r(A,C;B,D) \cdot A^{-1}.
\end{aligned}
\]
therefore, we can say that the points, $c_r(A,C;B,D)$ and $c_r(A^{-1},B^{-1}; C^{-1},D^{-1})$
are \emph{conjugatet-points} in a line of Desargues affine plane.
\qed

\begin{corollary}
If the point $A\in z[K]$ (center of skew field $K=(\ell^{OI},+,\cdot)$), then, 
\[c_r(A,C;B,D)=c_r(A^{-1},B^{-1}; C^{-1},D^{-1}). \]
\end{corollary}
\proof  If $A\in z[K]$ then, $AX=XA, \forall X\in K$, so $AXA^{-1}=X, \forall X\in K$. So, for $A\in z[K]$, we have that,
\[A\cdot c_r(A,C;B,D) \cdot A^{-1}=c_r(A,C;B,D).
\]
Hence
\[
c_r(A^{-1},B^{-1}; C^{-1},D^{-1})=c_r(A,C;B,D) \Leftrightarrow \text{if $A\in z[K]$}
\]
\qed
\begin{corollary}
In Papus affine plane, $c_r(A,C;B,D)=c_r(A^{-1},B^{-1}; C^{-1},D^{-1})$.
\end{corollary}

\begin{theorem}
If $A,B,C,D$ are distinct points in a line, in Desargues affine plane and $I$ is unitary point for multiplications of points in same line, have,
\[
c_r(A,B; C,D) \neq c_r(B,A;D,C)
\]
so, $c_r(A,B; C,D)$ is different point from $c_r(B,A;D,C)$.
\end{theorem}
\proof
From Definition of Cross-Ratio we have,
\[c_r(A,B;C,D)=\left[(A-D)^{-1}(B-D)\right]\left[(B-C)^{-1}(A-C)\right]=r(B,A;D)\cdot r(A,B;C)
\]
and
\[c_r(B,A;D,C)=\left[(B-C)^{-1}(A-C)\right]\left[(A-D)^{-1}(B-D)\right]=r(A,B;C) \cdot r(B,A;D)
\]
We mark the points, like below
\[R_1=r(A,B;C) \quad \text{and} \quad R_2=r(B,A;D)\]
so
\[c_r(A,B;C,D)=R_2 \cdot R_1 \quad \text{and} \quad c_r(B,A;D,C)=R_1 \cdot R_2
\]

This points are in $\ell^{OI}--$line in Desargues affine plane, so are elements of the skew fields $K=(\ell^{OI},+,\cdot)$, which are constructet over this line, so $E,F,G,H \in K$. So we have,
\[R_2 \cdot R_1 \neq R_1 \cdot R_2 \Rightarrow c_r(A,B;C,D) \neq c_r(B,A;D,C)
\]
\qed

\begin{corollary}
If $A,B,C,D \in \ell^{OI}$ are distinct points in a line, 
in Pappus affine plane and $I$ is unital point 
for multiplications, then 
\[
c_r(A,B; C,D)=c_r(B,A;D,C)
\]
\end{corollary}
\proof
If affine plane is Pappian plane, then the skew-field $(\ell^{OI},+,\cdot)$ is commutative, then is a Field.
\qed

We marked with $K=(\ell^{OI},+,\cdot)$ the skew field over $\ell^{OI}-$line in Desargues affine plane, we know that the center of the skew field $z[K]$, is a sub-skew field of $K$, moreover, $z[K]$ it is also commutative. 

\begin{theorem}
If $A,B,C,D \in \ell^{OI}$ are distinct points in a line, 
in Desargues affine plane and $I$ is unital point 
for multiplications of points in same line, then equation 
\[
c_r(A,B; C,D) = c_r(B,A;D,C)
\]
it's true, if\\
\begin{description}
	\item[(a)] points $A,B,C,D$ are in 'center of skew-field' $z[K]$;
	\item[(b)] ratio-points $r(A,B;C)$ are in 'center of skew-field';
	\item[(c)] ratio-point $r(B,A;D)$ are in 'center of skew-field';
	\item[(d)] ratio-point $r(A,B;D)$ is in centaralizer of point $r(A,B;C)$, or vice versa.
\end{description}

\end{theorem}
\proof
(a) If points $A,B,C,D \in z[K]$, we have that,
\[
A-D, B-D, B-C, A-C \in z[K] \]
and
\[ (A-D)^{-1}, (B-D)^{-1}, (B-C)^{-1}, (A-C)^{-1} \in z[K]
\]
also the production is commutative. Hence,
\[
[(A-D)^{-1}(B-D)] \cdot [(B-C)^{-1}(A-C)]=[(B-C)^{-1}(A-C)] \cdot [(A-D)^{-1}(B-D)]
\]
so,
\[ c_r(A,B;C,D)=c_r(B,A;D,C). \]

(b) If ratio-points $r(A,B;C)$ are in 'center of skew-field', we have that,
\[ X\cdot r(A,B;C)=r(A,B;C) \cdot X,\quad \forall X \in K \text{(so for all points $X\in \ell^{OI}$)}
\] 
so the equation is also true for the ratio-point $r(B,A;D)$, and have,
\[
\begin{aligned}
r(A,B;C) \cdot r(B,A;D) &= r(B,A;D) \cdot r(A,B;C)\\
[(B-C)^{-1}(A-C)]\cdot [(A-D)^{-1}(B-D)]&=[(A-D)^{-1}(B-D)]\cdot [(B-C)^{-1}(A-C)]\\
c_r(B,A;D,C)&=c_r(A,B;C,D).
\end{aligned}
\]
(c) in the same way, as in case (b).\\

(d) The Centralizer $\mathcal{C}_K(r(A,B;C))=\{ Y\in K | Y\cdot r(A,B;C)=r(A,B;C) \cdot Y\}$, and we have that, $r(B,A;D)\in \mathcal{C}_K(r(A,B;C))$, so we have,
\[
r(B,A;D) \cdot r(A,B;C)=r(A,B;C) \cdot r(B,A;D)
\]
so,
\[
c_r(A,B;C,D)=c_r(B,A;D,C),
\]
in the same way, it is proved that if $r(A,B;C) \in \mathcal{C}_K(r(B,A;D))$, then  $c_r(A,B;C,D)=c_r(B,A;D,C).$
\qed

\bibliographystyle{amsplain}
\bibliography{RCRrefs}

\end{document}